\documentclass[reqno,12pt]{amsart}

\numberwithin{equation}{section}

\usepackage{amssymb}
\usepackage{mathrsfs}
\usepackage{enumerate, xspace}

\usepackage{changebar}
\usepackage{hyperref}
\usepackage{pdfsync}
\usepackage{bm}
\usepackage{amsmath}
\usepackage{color}

\usepackage{amsfonts}
\usepackage{amsthm, upref}
\usepackage{graphicx}
\usepackage{bbm}

\newtheorem{thm}{Theorem}[section]
\newtheorem{lem}[thm]{Lemma}
\newtheorem{cor}[thm]{Corollary}

\newtheorem{prop}[thm]{Proposition}

\newtheorem{rem}[thm]{Remark}

\newcommand\cA{{\mathcal A}}

\newcommand\cC{{\mathcal C}}
\newcommand\cD{{\mathcal D}}

\newcommand\cL{{\mathcal L}}

\newcommand\bE{{\mathbb E}}
\newcommand\bN{{\mathbb N}}
\newcommand\bP{{\mathbb P}}
\newcommand\bR{{\mathbb R}}
\newcommand\bX{{\mathbb X}}

\newcommand{\nn}{\mathbb{N}}
\newcommand{\zz}{\mathbb{Z}}
\newcommand{\eq}{\begin{equation}}
\newcommand{\en}{\end{equation}}

\newcommand{\la}{\lambda}

\newcommand{\abbr}[1]{{\sc\lowercase{#1}}}

\begin{document}

\title[Convergence of the infinite Atlas process]
{The infinite Atlas process:\\ Convergence to equilibrium}

\author[A. Dembo]{Amir Dembo$^\star$}
\address{Amir Dembo\\
Stanford University, Stanford 94305, USA}
\email{{\tt adembo@stanford.edu}}

\author[M. Jara]{Milton Jara$^\dagger$}
\address{Milton Jara\\
IMPA, Estrada Dona Castorina, 110\\
Jardim Bot\^anico, CEP 22460-320
Rio de Janeiro, Brazil}
\email{{\tt mjara@impa.br}}

\author[S. Olla]{Stefano Olla$^\ddagger$}
\address{Stefano Olla\\
CEREMADE, UMR CNRS\\
Universit\'e Paris-Dauphine, PSL Research University\\
75016 Paris, France}
\email{{\tt olla@ceremade.dauphine.fr}}

\begin{abstract}
The semi-infinite Atlas process is a one-dimensional system 
of Brownian particles, where only the leftmost particle 
gets a unit drift to the right. Its particle spacing 
process has infinitely many stationary measures, with 
one distinguished translation invariant reversible measure.
We show that the latter is attractive for a large class 
of initial configurations of slowly growing (or bounded) 
particle densities. Key to our proof is a new estimate 
on the rate of convergence to equilibrium for the particle 
spacing in a triangular array of finite, large size systems.
\end{abstract}
\thanks{$\!\!\!\!\!\!\!\!\!\!$
$^\star$Research partially supported by NSF grant \#DMS-1613091.\\
$\dagger$Research partially supported by the ERC Horizon 2020  
grant \#715734, by  ANR-14-CE25-0011 grant EDNHS and by ANR-15-CE40-0020-01 grant LSD.\\
$\ddagger$Research partially supported by ANR-15-CE40-0020-01 grant LSD.\\}
\subjclass[2010]{60K35, 82C22, 60F17, 60J60}
\keywords{Interacting particles, reflecting Brownian motions,  
non-equilibrium hydrodynamics, infinite Atlas process}
\date{\today 
} 
\maketitle

\section{Introduction}\label{sec:intro}

Systems of competing Brownian particles interacting through their rank
dependent drift and diffusion coefficient vectors have received much 
recent attention (for example, 
in stochastic portfolio theory, where they 
appear under the name first-order market model, see \cite{fk} and the 
references therein).
For a fixed number of particles $m\in\nn$, such system 
corresponds to the unique weak solution of
\eq\label{mainSDE}
{d}Y_i(t)=\sum_{j \ge 1} \gamma_j\,{\bf 1}_{\{Y_i(t)=Y_{(j)}(t)\}}\,{d}t + \,
\sum_{j \ge 1} \sigma_j\,{\bf 1}_{\{Y_i(t)=Y_{(j)}(t)\}}
{d}W_i(t)\,,
\en
for $i=1,\ldots,m$, where $\underline{\gamma}=(\gamma_1,\ldots,\gamma_m)$ 
and $\underline{\sigma}=(\sigma_1,\ldots,\sigma_m)$ are
some constant drift and diffusion coefficient vectors and $(W_i(t), t \ge 0)$,
$i \ge 1$ are independent standard Brownian motions.
Here $Y_{(1)}(t)\leq Y_{(2)}(t)\leq\ldots\leq Y_{(m)}(t)$
are the ranked particles at time $t$ and the  
${\bR}_+^{m-1}$-valued \emph{spacings process}
$\underline{Z}(t)=(Z_1 (t), Z_2 (t), \dots,Z_{m-1}(t))$,
$t \geq 0$, is given by
\begin{equation}\label{gap}
Z_k(t) :=
X_{k+1}(t) - X_k(t) 
:= Y_{(k+1)} (t)- Y_{(k)}(t)
\,, \qquad k \ge 1 \,.
\end{equation}
The variables $X_k(\cdot)$ and $Z_k(\cdot)$ correspond to
the $k$-th ranked particle and $k$-th spacing, respectively,
with $X_1=\min_i Y_i$ denoting the leftmost particle. In 
particular, existence and uniqueness of such weak solution to 
\eqref{mainSDE} has been shown already in \cite{bp}. The
corresponding ranked process $\underline{X}(t)$ solves the system 
\begin{equation}
  \label{eq:rbm-gen}
  dX_j(t) = \gamma_j dt + \sigma_j d B_j (t) + dL_{j-1}(t) -
    dL_{j}(t) \,, \qquad j=1, \ldots, m
\end{equation}
for independent standard Brownian motions $(B_j(t), t \ge 0)$,
where $L_{j}(t)$ denotes the local time at zero of the 
non-negative semi-martingale $Z_j(\cdot)$, during $[0,t]$,
with $L_{0} \equiv 0$ and $L_{m} \equiv 0$. 
The spacing process $\underline{Z}(t)$ is 
thus a 
reflected Brownian motions (\abbr{rbm}) in a polyhedral domain. That is, the solution in ${\bR}_+^{m-1}$ of
\begin{equation}
  \label{eq:4-gen}
 dZ_j = (\gamma_{j+1} - \gamma_j) dt + 
 \sigma_{j+1} d B_{j+1}  - \sigma_j d B_{j} +
 2dL_{j} - dL_{j+1} - dL_{j-1} \,.  
\end{equation}
The general theory of
such \abbr{rbm} (due to \cite{hw,wil}, c.f. 
the survey \cite{Wil1995}), characterizes 
those $(\underline{\gamma},\underline{\sigma})$ 
for which the stationary distribution of $\underline{Z}(t)$
is a product of exponential random variables. Further utilizing 
this theory, \cite{sar1} deduces verious stochastic comparison 
results, whereas \cite{ips} and the references therein, 
estimate the rate $t_m$, $m \gg 1$ of convergence 
in distribution of the spacing process $\underline{Z}(t)$. 
In particular, the \textit{Atlas model} of $m$ particles, 
denoted by \abbr{Atlas}$_m(\gamma)$ (or \abbr{Atlas}$_m$ when 
$\gamma=1$), corresponds to 
$\sigma_j \equiv 1$ and $\gamma_j = \gamma {\bf 1}_{j=1}$. For
\abbr{Atlas}$_m(\gamma)$ it is shown in \cite[Corollary 10]{pp} 
that the spacing process $\underline{Z}(t)$ 
has the unique invariant measure
\begin{equation}\label{eq:rho-n}
\mu_\star^{(m,2\gamma)} := \bigotimes
\limits_{k=1}^{m-1}{\bf Exp} \big(2 \gamma (1-k/m) )\,, \quad m \in \nn \,,
\end{equation}
out of which \cite[Theorem 1]{pp} deduces that 
\begin{equation}
\label{rhola}
\mu_\star^{(\infty,2\gamma)} := \bigotimes\limits_{k=1}^{\infty}
{\bf Exp}(2\gamma),
\end{equation}
is an invariant measure for the spacings process of 
\abbr{Atlas}$_\infty(\gamma)$ (see also \cite{rv} for 
invariant measures of spacings when the particles
follow linear Brownian motions which are repelled by 
their nearest neighbors through a potential). By 
time-space scaling we hereafter set $\gamma=1$ \abbr{wlog}
and recall in passing that to rigorously construct the
\abbr{Atlas}$_\infty$ 
we call $\underline{y} = (y_i)_{i \ge 1} \in \bR^\infty$  \emph{rankable} 
if there exists a bijective mapping to the ranked terms
$y_{(1)} \le y_{(2)} \le y_{(3)} \le \ldots$
of $\underline{y}$. The 
solution of \eqref{mainSDE} starting at 
a fixed $\underline{y} \in \bR^{\infty}$ 
is then well defined if a.s.~the resulting process $\underline{Y}=(Y_1(t),Y_2(t),\ldots)$ is rankable at all $t$ 
(under some measurable ranking permutation). The \abbr{Atlas}$_\infty$ 
process is unique in law and well defined, when 
\begin{equation}
\label{maincond1}
{\mathbb P} \big( \sum_{i \ge 1} e^{-\alpha Y_i(0)^2} < \infty \big) = 1\,, 
\qquad \mbox{for any}\ \ \alpha > 0 
\end{equation}
(see \cite[Prop. 3.1]{sh2}), and \cite[Theorem 2]{iks} further 
constructs a strong solution of \eqref{mainSDE} in this setting
(more generally, whenever 
$\sigma_0^2=0$, 
$j \mapsto \sigma_{j+1}^2$ is 
concave and eventually, both $\sigma_j^2 =1$ and $\gamma_j=0$).
We focus here on  
\abbr{Atlas}$_\infty$, where \abbr{wlog} all \abbr{Atlas}$_m$,
$m \in \nn \cup \{\infty\}$, evolutions considered, 
start at a ranked configuration $\underline{Y}(0)=\underline{X}(0)$
having leftmost particle at zero (i.e. $X_1(0)=0$), and which
satisfies \eqref{maincond1}. For example, 
this 
applies when $\underline{Z}(0)$ is drawn from the product measure 
\begin{equation}\label{eq:rho-a-lambda}
\mu_\star^{(\infty,\lambda,a)} := \bigotimes_{k=1}^\infty {\bf Exp}
(\lambda + k a) \,, \qquad \lambda > 0, \, a \ge 0 \,.
\end{equation}
The natural conjecture made in \cite{pp} that $\mu_\star^{(\infty,2)}$
is the only invariant measure for spacing of \abbr{atlas}$_\infty$,
has been refuted by \cite{sar4} showing that 
$\{ \mu_\star^{(\infty,2,a)}$, $a \ge 0\}$, forms an \textit{infinite family of 
such invariant measures} 
(similar invariant spacings measures appeared earlier, in the 
non-interacting discrete model studied by \cite{ra,sh3}). 

As for the behavior of the leading particle of \abbr{Atlas}$_\infty$,
\cite{dembo2015equilibrium} verifies \cite[Conj. 3]{pp}, that 
starting with spacing at the translation invariant 
equilibrium law $\mu^{(\infty,2)}_\star$ results with 
\begin{equation}\label{eq:pp-conj}
t^{-1/4} X_{1} (t) \stackrel{d}{\rightarrow} {\mathcal N}(0,c) \,, \quad \textrm{ when }
\;\; t \to \infty \,, \;\; \textrm{ some } \; \; c \in (0,\infty) \,.
\end{equation}
Similar asymptotic fluctuations at equilibrium were established in
\cite{dgl,har} for a tagged particle in the \emph{doubly-infinite} 
Harris system (the non-interacting model with 
$\gamma_j \equiv 0$, $\sigma_j \equiv 1$), for the symmetric 
exclusion process associated with the \abbr{srw} on $\zz$
(starting with \cite{ar}),
and for a 
discrete version of the Atlas model (see \cite{hernandez2015equilibrium}).
In contrast, initial spacing of law $\mu^{(\infty,2,a)}_\star$ for $a>0$, 
induces a negative ballistic motion of the leading particle. Specifically,
\cite{tsai} and \cite{sar4} show respectively, that $\{X_{1}(t)+at\}$ is 
then a tight collection, of zero-mean variables.

Little is known about the challenging out of equilibrium behavior 
of \abbr{Atlas}$_\infty$. From \cite[Theorem 1]{pp} we learn that
at critical spacing density $\lambda=2$ the unit drift to the 
leftmost particle compensates the spreading of bulk particles 
to the left, thereby keeping the gaps at equilibrium. Such
interplay between spacing density and drift is re-affirmed by 
\cite{CDSS}, which shows that initial spacing law 
$\mu^{(\infty,\lambda)}_\star$ induces the a.s. convergence 
$t^{-1/2} X_1(t) \to \kappa$ 
with $\textrm{sign}(\kappa)=\textrm{sign}(2-\lambda)$. 
From \cite[Theorem 4.7]{sar2} it follows that if 
\abbr{Atlas}$_\infty$ starts at spacing law $\nu_0$ 
which stochastically dominates $\mu_\star^{(\infty,2)}$ (e.g. when
$\nu_0 = \mu_\star^{(\infty,\lambda)}$ any $\lambda \le 2$),
then the finite dimensional distributions (\abbr{fdd}) of 
$\underline{Z}(t)$ converge to those of $\mu_\star^{(\infty,2)}$ 
as $t \to \infty$. However, nothing else is known about 
the domain of attraction of $\mu_\star^{(\infty,2)}$ 
(or about those of $\{\mu_\star^{(\infty,2,a)}$, $a>0\}$). For 
example, what happens when $\nu_0 = \mu_\star^{(\infty,\lambda)}$
with $\lambda>2$? 

Our main result, stated next, answers this question by 
drastically increasing the established domain of attraction of
$\mu_\star^{(\infty,2)}$ spacing for \abbr{Atlas}$_\infty$.
\begin{thm}\label{thm:main}
Suppose the \abbr{Atlas}$_\infty$ process starts at 
$\underline{Z}(0)=(z_j)_{j \ge 1}$ such that for 
eventually non-decreasing $\theta(m)$
with $\inf_m \{\theta(m-1)/\theta(m)\} > 0$
and $\beta \in [1,2)$,  
\begin{align}
 \limsup_{m\to\infty} & \frac{1}{m^\beta \theta(m)} \sum_{j=1}^m  z_j \ < \infty\,, \qquad\qquad \label{eq:E2}\\ 
  \limsup_{m\to\infty} & \frac{1}{m^\beta \theta(m)} \sum_{j=1}^m (\log z_j)_- \ 
  < \infty\,, \qquad\qquad \label{eq:E1}\\  
\label{eq:E3}
\liminf_{m \to \infty} & \frac{1}{m^{\beta'} \theta(m)} \sum_{j=1}^m z_j = \infty \,, \qquad \beta':=\beta^2/(1+\beta), 
\end{align}
further assuming in case $\beta=1$ that $\theta(m) \ge \log m$.
Then, the \abbr{fdd} of $\underline{Z}(t)$ for the
\abbr{Atlas}$_\infty$ converge as $t \to \infty$ 
to  those of $\mu_\star^{(\infty,2)}$.
\end{thm}
For example, if $\lambda_j \in [c^{-1},c]$ and $\lambda_j z_j$ are i.i.d. 
of finite mean
such that $\bE (\log z_1)_- < \infty$, then taking $\theta(m) \equiv 1$ 
and $\beta>1$ small enough for $\beta'<1$, it follows by
the \abbr{slln} that \eqref{eq:E2}--\eqref{eq:E3} hold a.s.
Namely, when $\nu_0$ is any such product measure,  
$\underline{Z}(t)$ converges in \abbr{FDD} to 
$\mu_\star^{(\infty,2)}$. For independent 
$z_j \sim {\bf Exp}(\lambda_j)$ this applies
even when $\lambda_j \uparrow \infty$ slow enough so
$\sum_{j=1}^m \lambda_j^{-1}/(\sqrt{m} \log m)$ 
diverges (and hence \eqref{eq:E3} holds a.s.), or 
when $\lambda_j \downarrow 0$ slow enough to have
$m^{-\beta} \sum_{j=1}^m \lambda_j^{-1}$ 
bounded (for some $\beta<2$, so \eqref{eq:E2} holds).  
\begin{rem}
As matter of comparison, note that 
if $z_j$ decays to zero slower than $j^{-1/2} \log j$, 
then $\{z_j\}$ satisfies  
the hypothesis of Theorem \ref{thm:main} (for $\beta=1$), 
while for measures of 
the form \eqref{eq:rho-a-lambda}, generically $z_j$ decays like $j^{-1}$.
\end{rem}

The key to proving Theorem \ref{thm:main} is a novel control 
of the \abbr{Atlas}$_m$ particle 
spacing distance from equilibrium at time $t$, in terms of
the relative entropy distance 
of its initial law from equilibrium.
\begin{prop}\label{prop:finite}  
Start the \abbr{Atlas}$_m$ system at initial spacing law $\nu_0^{(m)}$
of finite entropy $H\big(\nu_{0}^{(m)}|\mu_\star^{(m,2)}\big)$
{and finite second moment $\int \|{\bf z}\|^2 d\nu_0^{(m)}$.}
Then, for any $t>0$ the spacing law $\nu^{(m)}_t$ is 
absolutely continuous with respect to the marginal 
of $\mu_\star^{(\infty,2)}$ and the
Radon-Nikodym derivative $g_t$ satisfies
\begin{equation}
  \label{eq:25n}
\int
  \Big\{
  \sum_{j=1}^{m-1} \big[(\partial_{z_{j-1}} 
      - \partial_{z_{j}}) \sqrt{g_t} \big]^2  \Big\} 
\prod_{j=1}^{m-1} 2 e^{-2 z_j} d z_j        
\le \frac{1}{2t}  H(\nu_0^{(m)} | \mu^{(m,2)}_\star)  + \frac 1m\,.
\end{equation}
\end{prop}
Combining Proposition \ref{prop:finite}
with Lyapunov functions for finite \abbr{Atlas} systems 
(constructed for example in \cite{dw,ips}), yields the 
following information on convergence of
the \abbr{Atlas}$_m$ particle spacing  \abbr{fdd}
at times $t_m \to \infty$ fast enough.
\begin{cor}\label{cor:fdd} 
Starting the \abbr{Atlas}$_m$ system at initial spacing law $\nu_0^{(m)}$
of finite second moment, for any fixed $k \ge 1$, 
the joint density of $(Z_1(t_m),\ldots,Z_k(t_m))$ with respect 
to the corresponding marginal of $\mu_\star^{(\infty,2)}$, 
converges to one, provided $t_m$ is large enough so both
$t_m^{-1} 
H\big(\nu_{0}^{(m)}|\mu_\star^{(m,2)}\big) \to 0$,
and $t_m^{-1} \sum_{j=1}^k Z_j(0) \to 0$ 
(in $\nu_0^{(m)}$-probability), as $m \to \infty$.
\end{cor}

\begin{rem}
For concreteness
we focused on the \abbr{Atlas}$_\infty$ process, but 
a similar proof applies for 
systems of competing Brownian particles 
where $\sigma_j^2 \equiv 1$,  $\gamma_1>0$ and $j \mapsto \gamma_j$ is 
non-increasing and eventually zero. 
We further expect this to extend to some of the 
two-sided infinite systems considered in \cite[Sec. 3]{sar-2sided},
and that such an approach may help in proving 
the attractivity of $\mu^{(\infty,2,a)}_\star$
in the \abbr{Atlas}$_\infty$ system. 
\end{rem}

In Section \ref{sec:finite-entr-init} we 
prove Proposition \ref{prop:finite} and Corollary \ref{cor:fdd},
whereas in Section \ref{sec:coupl-arbitr-inti}, we 
deduce Theorem \ref{thm:main} from Corollary \ref{cor:fdd} by
a suitable coupling of the \abbr{Atlas}$_m$
system and the left-most $k$ particles of \abbr{Atlas}$_\infty$ 
up to time $t_m$.

\section{Entropy control for \texorpdfstring{\abbr{Atlas}$_{m}$}{Atlasm}: 
Proposition \ref{prop:finite} and Corollary \ref{cor:fdd}}
\label{sec:finite-entr-init}


Recall \eqref{eq:rbm-gen}, which for \abbr{Atlas}$_{m}$ is
\begin{equation}
  \label{eq:rbm}
  X_j(t) = X_j(0) + {\bf 1}_{\{j=1\}} t + B _j (t) + L_{j-1} (t) -
    L_{j}(t) \qquad j=1, \ldots,m \,,
\end{equation}
where $L_j(t)$ denotes the local time on $\{s \in [0,t] : Z_j(s) = 0 \}$ 
for $1 \le j < m$, with $L_0(t) = L_{m}(t) \equiv 0$. 
Let $\bX_m := \{{\bf x} : x_1 \le x_2 \ldots \le x_m\} \subset \bR^m$.
The generator of the $\bX_m$-valued Markov process $\underline{X}(t)$ 
is then 
\begin{equation}
  \label{eq:5}
  (\widehat \cL_m g) ({\bf x}) := \frac 12 \sum_{j=1}^m \partial_{x_j}^2 +
   \partial_{x_1} 
\end{equation}
defined on the core of smooth bounded functions $g(\cdot)$ on
$\bX_m$ 
 satisfying
the Neumann boundary conditions
$$
\left(\partial_{x_j} - \partial_{x_{j+1}}\right)g \big|_{x_j=x_{j+1}} = 0,
\qquad j =1,\dots,m-1 \,.
$$  
Specializing \eqref{eq:4-gen} the corresponding $\bR_+^{m-1}$-valued spacings 
$Z_j(t) = X_{j+1}(t) - X_{j}(t)$
are then such that for $1 \le j \le m-1$,
\begin{equation}
  \label{eq:4}
 Z_j(t) =  Z_j(0) - {\bf 1}_{\{j=1\}} t + B_{j+1}(t)  -  B_{j}(t) +
 2L_{j}(t) - L_{j+1}(j) - L_{j-1}(t) \,.
\end{equation}
Let $\Delta^{(m)}$ denote the discrete Laplacian with 
Dirichlet boundary conditions at $0$ and $m$. Hence, 
using hereafter the convention of 
$\partial_{z_0} = \partial_{z_m} \equiv 0$,
\begin{equation}
  \label{eq:29}
    \Delta^{(m)} \partial_{z_j} := \partial_{z_{j-1}} -2 \partial_{z_j}
     + \partial_{z_{j+1}}  \,, \qquad j = 1,\ldots,m-1
    \,.
\end{equation}
Following this convention, in combination with the rule 
\begin{equation*}
    \partial_{x_j} = \partial_{z_{j-1}} - \partial_{z_j}\,,\qquad \qquad 
    j=1, \ldots,
    m\,,\qquad 
\end{equation*}
the generator of the $\bR_+^{m-1}$-valued Markov process
$\underline{Z}(t)$ is thus
\begin{equation}
  \label{eq:genfin}
    \cL_m = \frac 12  \sum_{j=1}^{m} \left(\partial_{z_{j-1}}
      - \partial_{z_{j}} \right)^2 -  \partial_{z_{1}} 
  =-\frac 12 \sum_{j=1}^{m-1}\partial_{z_{j}} (\Delta^{(m)}  \partial_{z_{j}} )  -  \partial_{z_{1}} 
\end{equation}
defined on the core $\cC_m$ of local, 
smooth functions $h({\bf z})$ such that  
\begin{equation}
  ( \Delta^{(m)}\partial_{z_{j}}) h \big|_{z_j = 0} = 0, 
  \qquad j=1,\ldots ,m-1\,. \qquad
\label{eq:core}
\end{equation}
Recall that $\mu_\star^{(m,2)}(\cdot)$ is the (unique) 
stationary law of $\underline{Z}(t)$ for \abbr{Atlas}$_m$.
In fact, for the density on $\bR_+^{m-1}$ of $\mu_\star^{(m,2)}(\cdot)$, 
\begin{equation}
  \label{eq:1}
  p_m ({\bf z}) := \prod_{j=1}^{m-1} \alpha_j e^{-\alpha_j z_j} \,, \quad
  \qquad \alpha_j :=  2 (1 - j/m) \,,
\end{equation}
a direct calculation shows that 
\begin{equation}\label{eq:ident-alpha}
 \frac 12 \sum_{j=1}^{m-1} \alpha_j \Delta^{(m)} \partial_{z_{j}} = -  \partial_{z_1} \,.
\end{equation}
Combining the \abbr{lhs} of \eqref{eq:genfin} with \eqref{eq:ident-alpha}
yields the symmetric form of the generator
\begin{equation}
  \label{eq:15}
    \cL_m = - \frac {1}{2p_m} 
     \sum_{j=1}^{m-1} \partial_{z_j} \big( p_m \Delta^{(m)} \partial_{z_{j}}   \big) 
\end{equation}
Using \eqref{eq:15} and integration by parts, we have 
 for bounded, smooth $g,h$ satisfying \eqref{eq:core}
\begin{align}
\int g (- \cL_m h) & d\mu^{(m,2)}_\star = 
\int h (- \cL_m g) d\mu^{(m,2)}_\star
\nonumber \\
&= \frac{1}{2} 
\int \Big\{ \sum_{j=1}^{m} [(\partial_{z_{j-1}} - \partial_{z_{j}}) g] [(\partial_{z_{j-1}} 
       - \partial_{z_{j}}) h] \Big\}  
d\mu^{(m,2)}_\star  := \cD_m (g,h) \,.
\label{eq:sym-form}
\end{align}
We see that $d\mu^{(m,2)}_\star = p_m({\bf z}) d {\bf z}$ is reversible 
for this dynamic, and the corresponding Dirichlet form 
$\cD_m (h) := \cD_m (h,h)$, extends from $\cC_m$, now only as
\begin{equation}
  \label{eq:2}
 \cD_m(h) = \frac 12 \int
  \Big\{\sum_{j=1}^{m}\left[\left(\partial_{z_{j-1}} 
      - \partial_{z_{j}} \right) h\right]^2 \Big\} 
p_m ({\bf z}) d {\bf z} \,,
\end{equation}
to the Sobolev space $W^{1,2}(\mu_\star^{(2,m)})$ 
of functions on $\bR^{m-1}_+$ with $L^2(\mu_\star^{(2,m)})$-derivatives.

We also consider the 
Markov dynamics on $\bR_+^{m-1}$ for the spacing 
process of an \abbr{Atlas}$_m$ whose $m$-th particle $X_m(s) = X_m(0)$
is frozen. Under the same convention as before,
the generator for this, right-anchored dynamics, is    
\begin{equation}
  \label{eq:genfinM}
   \widetilde \cL_m = \frac 12 \sum_{j=1}^{m-1} \left(\partial_{z_{j-1}} - \partial_{z_{j}} \right)^2  - \partial_{z_{1}}  
\end{equation}
for the core $\widetilde \cC_m$ of local, smooth functions $h({\bf z})$ 
such that  
\begin{equation}
  ( \widetilde\Delta^{(m)}\partial_{z_{j}}) h \big|_{z_j = 0} = 0, 
  \qquad j=1,\ldots ,m-1\,, 
\label{eq:core2}
\end{equation}
where $\widetilde\Delta^{(m)}$ is the discrete Laplacian with 
mixed boundary conditions. 
Specifically,
\begin{equation}\label{eq:29b}
    \widetilde\Delta^{(m)} \partial_{z_j} = \partial_{z_{j-1}} -
    2 \partial_{z_j} + \partial_{z_{j+1}}\,,
     \;\; 1 \le j \le m-2 \,, \quad
    \widetilde\Delta^{(m)} \partial_{z_{m-1}} = 
    \partial_{z_{m-2}} - \partial_{z_{m-1}} 
\end{equation}

For the remainder of this section we identify 
$\mu_\star^{(\infty,2)}$ with its marginal on 
${\bf z} = (z_1,\ldots,z_{m-1})$, whose density on $\bR_+^{m-1}$ is  
\begin{equation}\label{dfn:qstar}
q_m ({\bf z}) := \prod_{j=1}^{m-1} 2 e^{-2 z_j} \,.
\end{equation} 
Analogously to \eqref{eq:15} we find that 
\begin{equation}\label{eq:15-modified}
\widetilde \cL_m = - \frac{1}{2q_m} 
 \sum_{j=1}^{m-1} \partial_{z_j} \big(q_m \widetilde\Delta^{(m)}  \partial_{z_{j}} \big)  \,.
\end{equation}
This (marginal of) $\mu^{(\infty,2)}_\star$ is 
{thus reversible (stationary} 
and ergodic) for the right-anchored dynamics, and similarly to \eqref{eq:sym-form}--\eqref{eq:2} for 
bounded, smooth
$h$ satisfying \eqref{eq:core2}, the
associated Dirichlet form is given 
(on $W^{1,2}(q_m d{\mathbf z})$) by
\begin{equation}
  \label{eq:2M}
  \widetilde \cD_m (h) = \frac 12 \int \Big\{ 
  \sum_{j=1}^{m-1} \big[(\partial_{z_{j-1}}  - \partial_{z_{j}}) h\big]^2 
    \Big\} q_m({\bf z}) d{\bf z}       
      \,.
\end{equation}
Indeed, this reversible measure corresponds to starting the right-anchored
dynamics with $X_1(0)=0$ and 
a Gamma($2,m-1$) law for the frozen $X_m$, with the 
remainder $m-2$ initial particle positions chosen independently and 
uniformly on $[0,X_m]$.

\smallskip
\begin{proof}[Proof of Proposition \ref{prop:finite}]
Fixing $m \ge 2$, we 
start the finite particle dynamics of generator 
$\widetilde \cL_m$ of \eqref{eq:15}, with initial law 
$\nu_0^{(m)}$ on $\bR_+^{m-1}$ whose density  
\begin{equation}
  \label{eq:24}
f_0 := \frac{d\nu_0^{(m)}}{d\mu^{(m,2)}_\star} \,,    
\end{equation}
has the finite entropy 
\begin{equation}
H(\nu_0^{(m)} |\mu_\star^{(m,2)}) 
=  \int [f_0 \log f_0] ({\bf z}) p_m ({\bf z}) d {\bf z} =: H_m(f_0) \,.
\label{eq:23}
 \end{equation}
Recall \cite{amb-sav} that a 
Wasserstein solution of the Fokker-Planck equation
\begin{equation}\label{eq:ss}
     \partial_t f_t =  \cL_m f_t  \,,     
\end{equation}
starting at $f_0$, is a continuous (in the topology of weak convergence) collection of probability measures 
$t \mapsto f_t \mu_\star^{(2,m)}$ such that 
for any $s$, the  
derivatives $(\partial_{z_{j-1}} - \partial_{z_j}) \sqrt{f_s}$
exist a.e. in $\bR_+^{m-1}$, with
\begin{equation}\label{eq:fin-df}
\int_0^t \cD_m(\sqrt{f_s}) ds <\infty \qquad \forall t<\infty 
\end{equation} 
and moreover 
for any fixed compactly supported
smooth function $\zeta(t,{\bf z})$ on $\bR_+ \times \bR_+^{m-1}$,
 \begin{equation}
    \label{eq:weak-fk}
    \int_0^\infty 
    \left\{
\cD_m(f_t,\zeta(t, \cdot))    
    - \int \partial_t \zeta(t,{\bf z}) f_t({\bf z}) p_m ({\bf z}) d{\bf z}  
    \right\} dt \ = \ 0 \,.
  \end{equation}
{By \cite[Theorem 6.6]{amb-sav}, the law $\nu_t^{(m)}$ that corresponds to a starting measure $\nu_0^{(m)}$
of finite entropy and finite second moment, is a
Wasserstein solution $\nu^{(m)}_t = f_t \mu_\star^{(m,2)}$ 
of \eqref{eq:ss}.}\footnote{Though 
we could not find a reference for it, we
expect $f_t$ to be also a strong solution of \eqref{eq:ss}
which satisfies the boundary conditions of \eqref{eq:core}.}
From \cite[Theorem 6.6]{amb-sav} we further have that then 
$\sqrt{f_t} \in W^{1,2}(\mu_\star^{(2,m)})$ and
\begin{equation}
  \label{eq:7}
  \cD_m \left(\sqrt{f_t} \right) < \infty \qquad \forall t>0\,,
\end{equation}
with $t \mapsto  \cD_m \left(\sqrt{f_t} \right)$ non-increasing and
\begin{equation}
  \label{eq:6}
  H_m(f_t) - H_m(f_0) = - 4 \int_0^t \cD_m \left(\sqrt{f_s} \right) \; ds \,.
\end{equation}
Consequently, 
for any $t \ge 0$,
\begin{equation}
 4 t \,  {\cD}_m \left(\sqrt{f_t}\right) \le
 4 \int_0^{t} \cD_m\left(\sqrt{f_s}\right) ds =
  H_m(f_0) - H_m(f_t)  \le H_m(f_0)  \,.
\label{eq:27}
\end{equation}
Next, comparing \eqref{eq:1} with \eqref{dfn:qstar}, notice that
$q_m=p_m h_m$ for the strictly positive
\begin{equation*}
  h_m ({\bf z}) :=
  \prod_{j=1}^{m-1} \frac{2}{\alpha_j} e^{-\frac{2j}{m} z_j} \,,
\end{equation*}
such that 
\begin{equation}
  (\partial_{z_{j-1}} - \partial_{z_{j}}) \sqrt{h_m} =  
  \big(\frac 1m - {\bf 1}_{\{j=m\}} \big) \sqrt{h_m} \,.  
\end{equation}
Hence, for $f_t=h_m g_t$, using that $\sum_{j=1}^{m}
\left(\partial_{z_{j-1}} - \partial_{z_{j}}\right) =0$ and 
$\int g_t q_m d {\bf z} =1$, we arrive at  
\begin{align}
 2  \cD_m(\sqrt{f_t}) 
= & \int \Big\{ \sum_{j=1}^{m} \Big[\left(\partial_{z_{j-1}} 
- \partial_{z_{j}}\right) \sqrt{g_{t}} + \big(\frac{1}{m} - {\bf 1}_{\{j=m\}}
\big) \sqrt{g_{t}} 
\Big]^2  
\Big\} q_{m}  (\mathbf z) d{\mathbf z} \, \nonumber \\
= & 2 \widetilde \cD_m (\sqrt{g_t}) - \frac 1m + 
\int \big[\sqrt{g_t} - \partial_{z_{m-1}} \sqrt{g_t}\big]^2 q_m(\mathbf z) d{\mathbf z}
\,.
\label{eq:bdDF}
\end{align}
Combining \eqref{eq:27} and \eqref{eq:bdDF} we see that for any $t > 0$,
\begin{equation}
  \label{eq:25}
2 \widetilde \cD_{m} (\sqrt{g_{t}}) 
\le \frac{1}{2t}  H_m(f_0) + \frac 1m \,,
\end{equation}
where $g_{t}$ is precisely the density of $\nu^{(m)}_t$ 
with respect to the marginal of $\mu^{(\infty,2)}_\star$. 
In view of the definitions \eqref{eq:2M} and \eqref{eq:23}, 
the preceding bound
matches our claim \eqref{eq:25n}.
\end{proof}

With $\widetilde{D}_{k+1}(\sqrt{g})$ invariant to
mass-shifts $g({\bf z}) \mapsto \exp(2 \sum_{j=1}^{k} \theta_j) 
g({\bf z} - {\boldsymbol \theta})$, having 
$\widetilde{D}_{k+1}(\sqrt{g_m}) \to 0$ does not imply (uniform) 
tightness of the collection 
of probability measures
$\{ g_m q_{k+1} d {\bf z} \}_{m \in \bN}$. Instead, when proving 
Corollary \ref{cor:fdd}, we rely for tightness on the following
direct consequence of \cite[Sec. 3]{ips}.
\begin{lem}\label{lem:Atlask-tight}
For \abbr{atlas}$_{k+1}$, some $c_1=c_1(k)$ finite and
$D(t):=\sum_{j=1}^k Z_j(t)$,
\begin{equation}\label{eq:unif-tight}
\lim_{x \to \infty} \,
\sup_{t \ge c_1 D(0)} \, \{ \, \bP (D(t) \ge x) \, \} = 0 \,,
\end{equation}
where the supremum is also over all initial configurations.
\end{lem}
\begin{proof} Building on the construction in \cite[Sec. 3]{dw} 
of Lyapunov functions for \abbr{RBM} in polyhedral domains, 
while proving \cite[Thm. 3]{ips} the authors show that for 
the \abbr{Atlas}$_{k+1}$ (and more generally, for the 
spacing associated with \eqref{mainSDE}, whenever
$j \mapsto \gamma_j$ is non-increasing 
and $j \mapsto \sigma_j^2$ forms an arithmetic progression), one has
\begin{equation}\label{eq:lyapunov}
\bE [ V({\bf Z}(t))] \le e^{-t} [V({\bf Z}(0))] + c_2 \,, \qquad \forall t \ge 0 \,,
\end{equation}
where $V({\bf z})=e^{\langle {\bf v}, {\bf z} \rangle}$ for 
some strictly positive ${\bf v}$ and  
$c_2 < \infty$ (see \cite[inequality (51)]{ips}).
Noting that $\langle {\bf v},{\bf Z}(t) \rangle/D(t) \in [c_1^{-1},c_1]$
(with $c_1 := \max_j \{v_j \vee v_j^{-1}\}$), we get upon  
combining \eqref{eq:lyapunov} with Markov's inequality that 
for any initial configuration ${\bf Z}(0)$,
\begin{equation}\label{eq:exp-tail}
\bP(D(t) \ge x) \le 
\bP(\langle {\bf v}, {\bf Z}(t) \rangle \ge x/c_1 ) 
\le e^{- x/c_1} [ e^{-t} e^{c_1 D(0)} + c_2 ] \,.
\end{equation}
For $t \ge c_1 D(0)$ the \abbr{rhs} of \eqref{eq:exp-tail}
is at most $e^{-x/c_1} (1+c_2)$, yielding
\eqref{eq:unif-tight}.
\end{proof}
\begin{proof}[Proof of Corollary \ref{cor:fdd}] 
Fix probability densities $h_0 \ne h_1$ 
\abbr{wrt} the law $q_m d {\bf z}$ on $\bR_+^{m-1}$,
such that $\sqrt{h_0}$, $\sqrt{h_1} \in W^{1,2}(q_m d{\bf z})$.
Both properties then apply for $h_\la := \la h_1 + (1-\la) h_0$, any
$\la \in (0,1)$, and it is not hard to verify that 
\[
\frac{d^2 \widetilde \cD_m(\sqrt{h_\la})
}{d^2 \la}
 = \int 
\Big\{ \sum_{j=1}^{m-1} \big[ 
\sqrt{\alpha_0} (\partial_{z_{j-1}}-\partial_{z_j}) \sqrt{h_1} 
-
\sqrt{\alpha_1} (\partial_{z_{j-1}}-\partial_{z_j}) \sqrt{h_0} 
\big]^2 \Big\} \alpha_0 \alpha_1 q_m d {\mathbf z} \,,
\]
where the non-negative $\alpha_0:=h_0/h_\la$, $\alpha_1:=h_1/h_\la$ are  
uniformly bounded (per $\la$). Consequently, 
the map  
$h \mapsto \widetilde \cD_m(\sqrt{h})$
is convex on the set of probability densities $h$ with 
respect to the product law $q_m d {\bf z}$ on $\bR_+^{m-1}$.

The marginal density on $(z_1,\ldots,z_k)$ (\abbr{wrt} the $k$-th 
marginal of $\mu^{(\infty,2)}_\star$), is given for 
$\nu^{(m)}_t (d{\bf z}) = g_t q_m d{\bf z}$ and $1 \le k < m$, by 
\begin{equation}
  \label{eq:8}
  g_{t,k} (z_1, \ldots,z_k) := \int g_t({\bf z})
 \prod_{j=k+1}^{m-1} 2 e^{-2 z_j} dz_j \,.
\end{equation}
Thus, by the convexity of $\widetilde \cD_{m}(\sqrt{\cdot})$ and the formula
\eqref{eq:2M}, we have that 
\begin{align}
\widetilde \cD_m(\sqrt{g_t}) & = \int \widetilde \cD_m(\sqrt{g_t}) 
\prod_{j=k+1}^{m-1} 2 e^{-2 z_j} dz_j
\ge \widetilde \cD_m (\sqrt{g_{t,k}}) \ge \widetilde 
\cD_{k+1} (\sqrt{g_{t,k}}) \,.
\label{eq:convex-df}
\end{align}
In particular, fixing $k \ge 1$ and choosing $t_m \to \infty$
as in the statement of the corollary, we deduce from 
\eqref{eq:25n} and \eqref{eq:convex-df} that 
\begin{equation}\label{eq:app}
 \lim_{m\to\infty} \widetilde \cD_{k+1} \left(\sqrt{g_{t_m,k}}\right) = 0 \,.
\end{equation}
For $r \ge 2$ and the Markov generator $\widetilde \cL_r$ 
of \eqref{eq:genfinM} consider the functional 
on the collection of probability measures $\nu$ on $\bR_+^{r-1}$ 
defined by
\begin{equation}\label{eq:Ir}
\widetilde I_r (\nu) := \sup_{h \gg 0} \Big\{ \int
h^{-1} ( - \widetilde \cL_r h ) \, d \nu \Big\}\,,
\end{equation}
where the supremum is taken over all bounded 
away from zero, twice 
continuously differentiable functions having the 
boundary conditions \eqref{eq:core2} at $m=r$. With 
$h^{-1} (-\widetilde \cL_r h)$ then continuous and bounded,
clearly $\widetilde I_r(\cdot)$ is l.s.c. in the 
weak topology on probability measures in $\bR_+^{r-1}$. Further,
recall from \cite[Thm. 5]{DV1} that $\widetilde I_r(\nu) = \infty$ 
unless $\nu = g q_r d {\bf z}$ for a probability density $g$ such that 
$\sqrt{g} \in W^{1,2}(q_r d{\mathbf z})$, or equivalently
$\widetilde \cD_r(\sqrt{g})<\infty$, 
in which case $\widetilde I_r(\nu) = \widetilde \cD_r(\sqrt{g})$.
Hence, \eqref{eq:app} amounts to $\widetilde 
I_{k+1} (\nu^{(m,k)}_{t_m} ) \to 0$
for the joint law $\nu^{(m,k)}_{t_m}$ of $(Z_1(t_m),\ldots,Z_k(t_m))$
and any weak limit point of these laws must 
have a density $g$ \abbr{wrt} $q_{k+1} d {\bf z}$ 
such that $\widetilde \cD_{k+1} (\sqrt{g}) = 0$. From \eqref{eq:2M} it is thus 
necessarily that throughout $\bR^k_+$, 
\begin{equation*}
  \partial_{z_1} \sqrt{g} = 0, \qquad \left(\partial_{z_{j-1}}
     - \partial_{z_{j}} \right) \sqrt{g} = 0,\quad  j=2,\ldots,k\,,
\end{equation*}
so as claimed, the only possible limit point is $g \equiv 1$. 
By Prohorov's theorem, it remains to verify that
$\{ \nu_{t_m}^{(m,k)} \}$ are uniformly tight, namely, to provide
a uniform in $m$ tail-decay for $\sum_{j=1}^k Z_j(t_m)$ in the
corresponding \abbr{Atlas}$_m$ system. To this 
end, recall \cite[Cor. 3.10(ii)]{sar2} that under the same 
driving Brownian motions $\{B_j(s)\}$ and initial configuration,
the first $k$ spacings increase when 
all particles to the right of $X_{k+1}(0)$ are removed. 
Consequently, it suffices to provide a uniform in $m$ tail decay
for the diameter $D(t_m)$ of an \abbr{Atlas}$_{k+1}$ system 
of initial spacing distribution $\nu_0^{(m,k)}$. 
Fixing $\epsilon > 0$ we have from \eqref{eq:unif-tight} the 
existence of finite $c_1=c_1(k)$ and $x=x(\epsilon)$ such that 
for a given initial configuration, if $t \ge c_1 D(0)$ then 
$\bP(D(t) \ge x) \le \epsilon$. By our assumption that 
$t_m^{-1} \rho_m \to 0$ in $\nu^{(m,k)}_0$-probability, 
for the (random) initial diameter $\rho_m := X_{k+1}(0)-X_1(0)$, 
we have that 
$\bP(c_1 \rho_m \ge t_m) \le \epsilon$ for all $m \ge m_0(\epsilon)$,
in which case
\[
\bP(D(t_m) \ge x) \le 
\bP(D(t_m) \ge x, t_m \ge c_1 \rho_m) 
+ \bP(c_1 \rho_m \ge t_m) \le 2 \epsilon \,.
\]
With $\epsilon>0$ arbitrarily small and $x=x(\epsilon)$  
independent of $m$, we have thus established the required uniform tightness.
\end{proof}
\begin{rem} The proof of Proposition \ref{prop:finite} 
is easily adapted to deal with the right-anchored dynamic 
(of the generator $\widetilde \cL_m$ given in \eqref{eq:genfinM}). 
It yields for the latter dynamic the bound of \eqref{eq:25n}, now
with $(2t)^{-1} H(\nu_0^{(m)} | \otimes_{k=1}^{m-1} {\bf Exp}(2))$
in the \abbr{rhs}.
The proof of Lemma \ref{lem:Atlask-tight}  
also adapts to right-anchored dynamics,
hence the conclusion of Corollary \ref{cor:fdd} 
applies for sequences of right-anchored dynamics
when the latter expression decays to zero at $t=t_m \to \infty$
such that $t_m^{-1} \sum_{j=1}^k Z_j(0) \to 0$.
\end{rem}

\section{Coupling to \texorpdfstring{\abbr{Atlas}$_\infty$}{Atlas}: Proof of Theorem \ref{thm:main}}
\label{sec:coupl-arbitr-inti}

Let $\mathcal G(a) = (2\pi)^{-1/2} \int_a^{\infty} e^{-x^2/2} dx$
and consider the \abbr{Atlas}$_\infty$ process 
$\underline{Y}(t)=\{Y_i(t)\}$, denoting by  
$\underline{X}(t) = \{X_j(t)\}$ the corresponding 
ranked configuration. We first provide three elementary 
bounds for this process that are key to the proof of Theorem \ref{thm:main}.
\begin{lem}
  For any initial condition $\underline{X}(0)$, $\ell \ge 1$ and
  $t, \Gamma > 0$, 
  \begin{equation}
    \label{eq:3}
    \mathbb P\Big( \sup_{s \in [0,t]} \{X_1(s)\} \ge \Gamma \Big) \le
    2 \mathcal G\Big(\frac{\ell \, \Gamma  - t - \sum_{j=1}^{\ell}
        X_j(0)}{\sqrt{\ell t}}\Big) \,. 
  \end{equation}
  \begin{proof} Starting \abbr{wlog} 
at $\underline{Y}(0)=\underline{X}(0)$, we have that for any $s \ge 0$, 
    \begin{equation*}
      X_1(s) 
      \le 
      \frac 1\ell \sum_{i=1}^{\ell}  Y_i(s) \le   
      \frac {s}\ell + \frac
      1\ell \sum_{j=1}^{\ell} X_j(0) + \frac{1}{\sqrt{\ell}} {\widetilde W} (s) \,,
\end{equation*}
where ${\widetilde W}(s) := \ell^{-1/2} \sum_{i=1}^{\ell} W_i(s)$ is  
a standard Brownian motion. Thus, by the reflection principle, 
\begin{align*}
    \mathbb P\Big( \sup_{s \in [0,t]} \{X_1(s)\} \ge \Gamma \Big) &\le
    \mathbb P\Big( \sup_{s \in [0,t]} \{\frac {1}{\sqrt{\ell}}
      \widetilde W(s) \} \ge \Gamma - \frac {t}\ell - \frac 1\ell 
      \sum_{j=1}^{\ell}
      X_j(0) \Big)\\
    &= 2 \mathbb P\Big( \widetilde W(t) \ge {\sqrt \ell} \, \Gamma  - \frac
      {t}{\sqrt{\ell}}  - \frac 1{\sqrt \ell} \sum_{j=1}^{\ell}
      X_j(0) \Big) \,,
\end{align*}
which upon Brownian scaling yields the stated bound of \eqref{eq:3}.
  \end{proof}
\end{lem}
\begin{lem}\label{lem:top-control}
For $X_1(0) \ge 0$, $\Gamma$ and $k \ge 2$ such that
$\Gamma^{(k)} := (\Gamma - X_k(0))/3 > 0$, any
$\ell \ge 1$ and $t>0$, we have that  
  \begin{equation}
    \label{eq:18}
    \mathbb P\Big( \sup_{s \in [0,t]} \{X_k(s)\} \ge \Gamma \Big) \le
      2 \mathcal G\Big(\frac{\ell \, \Gamma^{(k)} -  t - \sum_{j=1}^{\ell}
        X_j(0)}{\sqrt{\ell t}}\Big) + 4 k \mathcal G
        \Big(\frac{\Gamma^{(k)}}{\sqrt t}\Big) \,.
  \end{equation}
\end{lem}

\begin{proof} Recall \cite[Cor. 3.12(ii)]{sar2} that keeping the same driving Brownian motions $\{B_j(s)\}$ and initial configuration
$\underline{X}(0)$, the spacing vector 
$\underline{Z}(t)$ is pointwise decreasing in $\gamma$. Further, by \cite[Cor. 3.10(ii)]{sar2}, the first $k-1$ spacings increase when all particles to
the right of
$X_{k}(0)$ are removed. Consequently, that value of $X_k(t)-X_1(t)$ at the drift $\gamma=1$ of \abbr{Atlas}$_\infty$ is bounded
by its value for a $k$-particle Harris system (of $\gamma=0$), starting 
at same positions as the original \abbr{Atlas}$_\infty$ 
process left-most $k$ particles. In the latter
case, letting $V_k(s) :=\max_{j=1}^k \{ B_j(s) \}$
and the identically distributed 
$V'_k(s):=\max_{j=1}^k \{-B_j(s)\}$, 
our assumption that $X_1(0) \ge 0$ results with
\begin{equation*}
X_k(s) - X_1(s) \le 
 X_k(0) + V_k(s) + V'_k(s) \,. 
\end{equation*}
Thus,  
with $\Gamma^{(k)} = (\Gamma - X_k(0))/3$ and $\{\widetilde B(s)\}$
denoting a standard Brownian motion, we get by the union bound that 
\begin{align*}
     \mathbb P\Big( \sup_{s \in [0,t]} \{X_k(s)\} \ge \Gamma \Big) &\le
    \mathbb P\Big( \sup_{s \in [0,t]} \{X_1(s) +  V_k(s) + V'_k(s)\}    
     \ge \Gamma - X_k(0)  \Big)  \\
    & \le \mathbb P\Big( \sup_{s \in [0,t]} \{X_1(s)\} \ge \Gamma^{(k)} \Big)  
    + 2 k \mathbb P\Big(  \sup_{s \in [0,t]} \{ \widetilde B(s)\}
        \ge \Gamma^{(k)} \Big) \,.
\end{align*}
Consequently, by \eqref{eq:3} and the reflection principle, 
\begin{equation*}
     \mathbb P\Big( \sup_{s \in [0,t]} \{X_k(s)\} \ge \Gamma \Big) \le
     2 \mathcal G\Big(\frac{\ell \, \Gamma^{(k)} - t - \sum_{j=1}^{\ell}
        X_j(0)}{\sqrt{\ell t}}\Big) + 4k \mathcal G \Big(
        \frac{\Gamma^{(k)}}{\sqrt{t}} \Big) \,,
\end{equation*}
as claimed.
\end{proof}

\begin{lem}\label{lem:bulk-control}
For any $m \ge 0$, $t,\Gamma >0$ and initial configuration 
$\underline{Y}(0)=\underline{X}(0)$,
  \begin{equation}
    \label{eq:19}
    \mathbb P\Big( \inf_{s \in [0,t]}\; \inf_{i > m} \{Y_i(s)\} \le \Gamma
    \Big) \le 2 \sum_{i>m} \mathcal G\Big(\frac{ X_i(0)-
        \Gamma}{\sqrt t} \Big) \,. 
  \end{equation}
\end{lem}
\begin{proof} Removing the drift in the \abbr{Atlas} model decreases 
all coordinates of $\underline Y(s)$
and correspondingly increases the \abbr{LHS} of \eqref{eq:19}. Thus, 
\begin{align*}
\mathbb P\Big(\inf_{i > m} \inf_{s \in [0,t]} \{Y_i(s)\}  \le \Gamma
\Big) &\le \sum_{i>m} \mathbb P \Big(\inf_{s \in [0,t]}
\{W_i(s)\} \le \Gamma - X_i(0)\Big) 
\\
&= 2 \sum_{i>m} \mathbb P\Big( W(1) \le \frac{\Gamma-X_i(0)}{\sqrt t} 
\Big)
=2 \sum_{i>m} \Big(\frac{X_i(0)- \Gamma}{\sqrt t} \Big)\,,
\end{align*}
as claimed. 
\end{proof}

\begin{proof}[Proof of Theorem \ref{thm:main}.]
Given initial spacing configuration $\underline{z}$ 
that satisfies \eqref{eq:E2} and \eqref{eq:E1}, consider 
the following two sequences 
of initial distributions for the finite increment vectors  
$\underline{Z}_m (0) := (Z_1(0),\ldots,Z_{m-1}(0))$ of 
the \abbr{Atlas}$_m$ process, $m \ge 2$.
Starting at the measure $\nu_0^{(m,-)}(\cdot) = \mu^{(m,2)}_\star 
( \cdot | \underline{Z}_m (0) \le \underline{z}_m)$
for the given $\underline{z}_m = (z_1,\ldots,z_{m-1})$ 
yields for same driving Brownian motions 
an \abbr{Atlas}$_m$ spacing process 
$\underline{Z}_m^- (s)$ which is dominated at all times $s \ge 0$ 
by the corresponding process that started 
at spacing $\underline{z}_m$, whereas 
$\nu_0^{(m,+)}(\cdot) = \mu^{(m,2)}_\star (\cdot |
\underline{Z}_m(0) \ge \underline{z}_m)$ similarly 
yields a spacing process $\underline{Z}^+_m(s)$ 
that dominates the spacing for the original 
process which started at $\underline{z}_m$. 
The corresponding relative entropies are then
\begin{align}\label{eq:ent-umeas}
H_m^+ & := H(\nu_0^{(m,+)}|\mu^{(m,2)}_\star) =
-\log \mu_\star^{(m,2)} (\{ \prod_{j=1}^{m-1} [z_j,\infty)\} ) 
= \sum_{j=1}^{m-1} \alpha_j z_j \le 2 \sum_{j=1}^{m-1} z_j \\
H_m^- & := H(\nu_0^{(m,-)}|\mu^{(m,2)}_\star) =-\log \mu_\star^{(m,2)} 
(\{\prod_{j=1}^{m-1} [0,z_j]\}) 
= \sum_{j=1}^{m-1} -\log (1-e^{-\alpha_j z_j}) \nonumber \\
& \le \sum_{j=1}^{m-1} \big[ 1 + (\log \alpha_j z_j)_- \big] 
\le 2 m \log m + \sum_{j=1}^{m-1} (\log z_j)_- \,,
\label{eq:ent-lmeas}
\end{align}
since $-\log (1-e^{-u}) \le 1 + (\log u)_-$ 
for all $u \ge 0$, while $\alpha_j \ge 2/m$ (see
\eqref{eq:1}), hence $\log (e/\alpha_j) \le 2 \log m$.
By \eqref{eq:E2} and \eqref{eq:ent-umeas},
\[
\limsup_{m \to \infty} \frac{H(\nu_0^{(m,+)}|\mu_\star^{(m,2)})}{m^\beta \theta (m)} < \infty \,.
\] 
With $\theta(m) \ge \log m$ in case $\beta=1$, we 
similarly deduce from \eqref{eq:E1} and \eqref{eq:ent-lmeas} that 
\[
\limsup_{m \to \infty} \frac{H(\nu_0^{(m,-)}|\mu^{(m,2)}_\star)}{m^\beta 
\theta (m)} < \infty \,.
\] 
Fixing $k \ge 2$, the uniform 
over $m \ge 2k$ first moment bound,
\[
\nu_0^{(m,-)} \big[ \sum_{j=1}^k Z_j (0) \big] \le
\nu_0^{(m,+)} \big[ \sum_{j=1}^k Z_j (0) \big] 
 = \sum_{j=1}^{k} (z_j + \alpha_j^{-1})  
\le \sum_{j=1}^{k} (z_j + 1) \,,
\]
yields by Markov's inequality 
that $t_m^{-1} \sum_{j=1}^k Z_j(0) \to 0$ in 
$\nu^{(m,\pm)}_0$-probability, 
for any $t_m \to \infty$.
Further, the 
second moment of $\nu_0^{(m,-)}$ is finite (being at most
$\|\underline z\|^2$), as is the second moment of $\nu_0^{(m,+)}$ 
(being at most the product of $e^{H_m^+}$ and 
the finite second moment of $\mu_\star^{(m,2)}$).
For $\psi(m) \uparrow \infty$, with
$\inf_m \{ \psi(m-1)/\psi(m) \} > 0$, let  
\begin{equation}\label{dfn:tm}
t_m := 2 m^\beta \theta(m) \psi(m) \,.
\end{equation}
Setting $m_n=m_n({\bf s}):=\inf\{m \ge 2 : t_m \ge s_n\}$,
our constraints on $\theta(\cdot)$ and $\psi(\cdot)$ yields 
for $t_m$ of \eqref{dfn:tm} and \emph{any $s_n \uparrow \infty$,} 
\[
\inf_{n \ge 1} 
\Big\{ \frac{s_n}{t_{m_n}} \Big\} 
\ge \inf_{m \ge 2} \Big\{ \frac{t_{m-1}}{t_m} \Big\} > 0 \,.
\]
Thus, by the preceding, upon applying Corollary \ref{cor:fdd} to the 
\abbr{Atlas}$_m$ model initialized at $\nu_0^{(m,\pm)}$
we have that the joint law of the first $k$ coordinates of 
$\underline{Z}_{m_n}^{\pm}(s_n)$,
converges as $n \to \infty$ to the corresponding marginal 
of $\mu^{(\infty,2)}_\star$. 
The same limit in distribution then
applies for the spacing 
$(Z_1(s_n),\ldots,Z_k(s_n))$ of \abbr{Atlas}$_{m_n}$
started at $\underline{z}_{m_n}$ (which is 
sandwiched between the corresponding marginals of 
$\underline{Z}_{m_n}^{-}(s_n)$ and 
$\underline{Z}_{m_n}^{+}(s_n)$).
Assuming further that $\sup_m \{\psi(m)/m\} \le 1$, 
we claim that for
$X_1(0)=0$ and the given initial spacing $\underline{Z}(0)=\underline{z}$, 
the \abbr{RHS} of \eqref{eq:18} is summable over $m$, 
at $t=t_m$ of \eqref{dfn:tm} and
\begin{equation}\label{dfn:Gamma}
\Gamma_m := 36 m^{\beta'} \theta(m) \psi(m)^{\beta/(1+\beta)} \,, \quad
\ell_m := [m^{\beta/(1+\beta)} \psi(m)^{1/(1+\beta}] \,.
\end{equation}
Indeed, since $\theta(\cdot)$ is eventually non-decreasing, we have 
then that  
\begin{equation}\label{eq:set-gamma}
\frac{1}{12} \Gamma_m \ell_m \ge t_m \ge \ell_m^{1+\beta} \theta(\ell_m) \,,
\qquad \forall m \ge m_\star 
\end{equation} 
Further, with $k$ fixed and $\Gamma_m \uparrow \infty$, 
necessarily $X_k(0) \le \Gamma_m/8$ for all $m \ge m_\star$ 
large enough, in which case from \eqref{eq:set-gamma}, 
the \abbr{RHS} of \eqref{eq:18} is bounded above 
for such $m$, $t_m$, $\Gamma_m$ and $\ell_m$, by 
\[
2 \mathcal G\Big(\sqrt{\ell_m^{\beta} \theta(\ell_m)}\Big) + 
4 k \mathcal G \Big(3 \sqrt{\ell_m^{\beta-1} 
\theta(\ell_m)}\Big) 
\]    
and recalling that $\theta(\ell_m) \ge \frac{1}{2} \log m$ when $\beta=1$, 
it is easy to verify that the preceding bound is summable in $m$.
Next, utilizing \eqref{eq:E3}, we can further make sure that
$\psi(m) \uparrow \infty$ slowly enough so that for any fixed
$\kappa<\infty$, 
\begin{equation}\label{eq:Xi-lbd}
X_m(0) \ge (\kappa+1) \Gamma_m\,, \qquad \forall m \ge m_\kappa 
\end{equation} 
so that for $m \ge m_\kappa$ the \abbr{RHS} of \eqref{eq:19} 
is bounded above, at $t=t_m$ and $\Gamma=\Gamma_m$, by
\begin{equation}\label{eq:E4}
2 \sum_{j=1}^\infty \mathcal G\Big(\kappa \Gamma_{m+j}/\sqrt{t_m}\Big) \,.
\end{equation}
Note that $\beta' \ge 1/2$ and $\delta := \beta/(1+\beta) - 1/2 \ge 0$
is strictly positive when $\beta>1$. Thus, 
$\beta'-\beta/2 = \delta \beta \ge 0$ and setting  
$\kappa' := (18 \kappa)^2$, we deduce from \eqref{dfn:tm} and 
\eqref{dfn:Gamma} that 
\[
\kappa \frac{\Gamma_m}{\sqrt{t_m}} \frac{\Gamma_{m+j}}{\Gamma_m} \ge 
m^{\beta \delta} \theta(m)^{1/2} \sqrt{2\kappa' (1+j/m)} \,.
\]
Increasing $m_\kappa$ if needed, we have that
$m^{2 \beta \delta} \theta(m) \ge \log m$ for all $m \ge m_\kappa$, in 
which case for $b_m:=m^{-1} \log m$, 
the expression \eqref{eq:E4} is further bounded by
\[
2 \sum_{j=1}^\infty \mathcal G\Big(\sqrt{2 \kappa' (1+j/m) \log m}\Big)  
\le 2 m^{-\kappa'} \sum_{j=1}^\infty e^{-\kappa' j b_m}
\le \frac{2}{\kappa' b_m} m^{-\kappa'}  
\]
(recall the elementary bound $\mathcal G(x) \le e^{-x^2/2}$ for $x \ge 0$).
Thus, for any $\kappa' > 2$, such choices of $t_m$ and 
$\Gamma_m$ guarantee that the \abbr{RHS} of \eqref{eq:19} is also 
summable over $m$. Combining Lemma \ref{lem:top-control},
Lemma \ref{lem:bulk-control} and the Borel-Cantelli lemma
we deduce that almost surely, the events 
\[
\cA_m : = \Big\{ 
\sup_{s \in [0,t_m]} \{X_k(s)\} < 
\Gamma_m \le 
\inf_{s \in [0,t_m], i > m} \, \{Y_i(s)\} \, 
\Big\} \,,
\] 
occur for all $m$ large enough.
Note that $\cA_m$ implies that throughout 
$[0,t_m]$ the $k$ left-most particles of 
the \abbr{Atlas}$_\infty$ process are from among the 
initially left-most $m$ particles. From this it follows 
that under $\cA_{m_n}$ the spacing $(Z_1(s_n),\ldots,Z_{k-1}(s_n))$ 
for the \abbr{Atlas}$_{m_n}$ coincide with those for the 
\abbr{Atlas}$_\infty$, when using the same 
initial configuration $\underline{z}$ and driving Brownian 
motions $\{W_i(s)\}$. Having proved already the 
convergence in distribution when $n \to \infty$, of 
the first $k-1$ spacing for \abbr{Atlas}$_{m_n}$ at time $s_n$ 
and that the events $\cA_m$ occur for all $m$ large enough, we conclude 
that the \abbr{fdd} of spacing for \abbr{Atlas}$_\infty$ 
converge to those of $\mu^{(\infty,2)}_\star$, along
any (non-random) sequence $s_n \uparrow \infty$, as claimed.
\end{proof}

\end{document}